\documentclass[11pt]{article}
\usepackage[margin=1in, a4paper]{geometry}
\usepackage{amsthm, amsfonts,amsmath,amssymb,bbm, bbold}
\usepackage{mleftright}
\usepackage{verbatim}
\usepackage{textgreek}
\usepackage[]{authblk}
\usepackage{color}
\usepackage{hyperref}
\usepackage{etoolbox}
\makeatletter
\let\ams@starttoc\@starttoc
\makeatother
\usepackage[parfill]{parskip}
\makeatletter
\let\@starttoc\ams@starttoc
\patchcmd{\@starttoc}{\makeatletter}{\makeatletter\parskip\z@}{}{}
\makeatother

\usepackage{enumerate}
\usepackage{pdfsync}
\usepackage{color}

\numberwithin{equation}{section}

\newtheorem{theorem}{Theorem}[section]

\newtheorem{lemma}[theorem]{Lemma}

\theoremstyle{definition}

\newtheorem{remark}[theorem]{Remark}

\newcommand{\lb}{\mleft(}
\newcommand{\rb}{\mright)}
\newcommand{\lbb}{\mleft [}
\newcommand{\rbb}{\mright ]}
\newcommand{\labs}{\mleft |}
\newcommand{\rabs}{\mright |}
\newcommand{\lbrb}[1]{\lb #1 \rb}

\newcommand{\labsrabs}[1]{\labs#1\rabs}

\newcommand{\langlerangle}[1]{\langle#1\rangle}
\newcommand{\lbcurly}{\mleft\{}
\newcommand{\rbcurly}{\mright\}}
\newcommand{\lbcurlyrbcurly}[1]{\lbcurly#1\rbcurly}

\newcommand{\simi}{\stackrel{\infty}{\sim}}

\newcommand{\bo}[1]{\mathrm{O}\lbrb{#1}}
\newcommand{\so}[1]{\mathrm{o}\lbrb{#1}}

\newcommand{\Ebb}[1]{\Eb\lbb #1\rbb}

\newcommand{\Eb}{\mathbb{E}}

\newcommand{\Nb}{\mathbb{N}}
\newcommand{\N}{\mathbb{N}}
\newcommand{\Rb}{\mathbb{R}}
\newcommand{\R}{\mathbb{R}}

\renewcommand{\P}{\mathbb{P}}

\newcommand{\Lc}{\mathcal{L}}

\newcommand{\ind}[1]{\mathbbm{1}_{\lbcurlyrbcurly{#1}}}

\newcommand{\Var}{\operatorname{Var}}

\usepackage{authblk}

\newcommand{\D}{\mathrm{d}}

\makeatletter
\newcommand{\myitem}[1]{%
	\item[#1]\protected@edef\@currentlabel{#1}%
}
\makeatother

\usepackage{tikz}
\usetikzlibrary{decorations.markings}
\usetikzlibrary{patterns}
\DeclareMathSymbol{\mrq}{\mathord}{operators}{`'}
\usepackage{wrapfig}

\allowbreak

\begin{document}
	
	\date{}
	\author[,1]{Martin Minchev \thanks{
			Email: mjminchev@fmi.uni-sofia.bg.}}
	\author[,1,2]{Maroussia Slavtchova-Bojkova  \thanks{ 
			Email: bojkova@fmi.uni-sofia.bg}
	}
	
	\affil[1]{Faculty of Mathematics and Informatics, Sofia University "St. Kliment Ohridski", 5,
		James Bourchier blvd., 1164 Sofia, Bulgaria}
	
	\affil[2]{Institute of Mathematics and Informatics,  Bulgarian Academy of Sciences, Akad.  Georgi Bonchev str., 
		Block 8, Sofia 1113, Bulgaria}
	
	\title{Multi-type branching processes with immigration generated by  point processes}
	\maketitle

	%\keywords{Asymptotic analysis, Functional equations, Exponential functional, L\'evy processes, Wiener-Hopf factorizations, Infinite divisibility,  Stable L\'evy processes, Special functions, Intertwining, Bernstein function} % Separate items with ;
	
	%\subjclass{30D05, 60G51, 60J55, 60E07, 44A60, 33E99}

	\maketitle

\begin{abstract}
	Following the pivotal work of Sevastyanov \cite{Sevastyanov-57}, who considered branching processes with homogeneous Poisson immigration, much has been done to understand the behaviour of such processes under different types of branching and immigration mechanisms. 
		
		Recently, the case where the times of immigration
are generated by a non-homogeneous Poisson process was considered in depth. 
		In this work, we  demonstrate how one can use the framework of
point processes in order to go
beyond the Poisson process. As an illustration, we show how to transfer techniques from the case of Poisson immigration to the case where it is spanned by a determinantal point process.
\end{abstract}

	% In this work, we consider various generalizations of the Poisson point process that generate immigration times in the case of single or multitype branching processes, following the approach of \cite{Bojkova}. In particular, we obtain some asymptotic results and discuss the challenges of using Cox processes, which exhibit clustering behavior, and determinantal processes, which, in contrast, exhibit repulsive properties. For background on these point processes, see, e.g., \cite[Chapter 13]{Last} and \cite{Decreusefond}.\\
	
	
	
	\section{Introduction}

The main purpose of this paper is to examine the impact of immigration on the behaviour of
both single- and multitype continuous-time Markov branching processes, focussing on immigration types that differ from the Poisson model. Although classical Poisson processes are the most commonly used to model the immigration component, they have certain limitations in practical applications. These include the assumption of equidispersion, which may not always hold,  and their time-homogeneity, which is addressed by introducing time-non-homogeneous intensity rates. Recently, such branching models with immigration at the jump times of a non-homogeneous Poisson process were considered in depth by Mitov et al. \cite{Mitov-Yanev-Hyrien-18-critical-multi} in the critical case, and later by Slavtchova-Bojkova et al. \cite{ Bojkova-Hyrien-Yanev-22, Bojkova-Hyrien-Yanev-23} in the non-critical case. In these studies, the efforts were concentrated mainly on analysing the asymptotic behaviour of the processes for various rates of the Poisson measure, under the assumption that these rates are asymptotically equivalent to either exponential or regularly varying functions. Consequently, results similar to the strong law of large numbers and central limit theorems were established.

Multitype Markov branching processes were first introduced by Kolmogorov and Dmitriev \cite{Kolmogorov-Dmitriev-47}, marking the term \emph{branching process} as one of the earliest concepts in the literature. The notion of branching processes with immigration was later formalised by Sevastyanov \cite{Sevastyanov-57}, who explored a single-type Markov process with immigration driven by a time-homogeneous Poisson process, and derived limiting distributions for subcritical, supercritical, and critical cases.

 Since then, numerous extensions of branching processes with immigration have been developed and thoroughly explored. Seminal reviews by Sevastyanov \cite{Sevastyanov-68}, and Vatutin and Zubkov \cite{Vatutin-Zubkov-87,Vatutin_AZ_93}, have highlighted many key results. More recent advancements have been made by Barczy et al. \cite{Barczy1,Barczy2}, González et al. \cite{Gonzalez-Kersting-Minuesa-19}, and Li et al. \cite{Li-Vatutin-Zhang-21}, among others, who have expanded the theory of these processes.

Branching processes with time-non-homogeneous
immigration were first introduced by Durham \cite{Durham-71} and Foster and Williamson \cite{Foster-Williamson-71}.
 Comprehensive reviews of these models can be found in the monographs by Badalbaev et al. \cite{BadRah} and Rahimov \cite{Rahimov-95}, as well as in Rahimov’s review paper \cite{Rahimov-21}. 
  In the recent work of Rabehasaina and Woo
\cite{Rabehasaina-21},  the model described in Mitov et al. \cite{Mitov-Yanev-Hyrien-18-critical-multi} and Slavtchova-Bojkova et al. \cite{Bojkova-Hyrien-Yanev-22, Bojkova-Hyrien-Yanev-23} was considered, and the established limit results were obtained by means of characteristic functions.

In this work, we aim to introduce the general framework of Laplace functionals of random point processes within the context of branching processes. We believe that this approach will be valuable for researchers looking to extend their work beyond the Poisson process. For example, our method simplifies the derivation of the generating function for a branching process with immigration, which often serves as the foundation for many studies. Traditionally, this derivation was based on the exact distribution of points over a given interval, which is tractable in the case of a Poisson process, as demonstrated in \cite[Theorem 1]{Mitov-Yanev-Hyrien-18-critical-multi}, \cite[Lemma 3.1]{Rabehasaina-21}, and \cite[Theorem 1]{Butkovsky-12}. Our framework provides a more unified and often simpler approach to addressing problems involving immigration, as one can see from the proof of the third item of Theorem \ref{thm: DPP2} and similar results for Poisson immigration, such as \cite[Theorem 8]{Mitov-Yanev-Hyrien-18-critical-multi} and \cite[Theorem 3]{Rabehasaina-21}.  Due to the tractability of the framework, we correct minor errors from the last two referenced results, see Remark \ref{rem: eigen}.

The paper is structured as follows:  in Section \ref{sec: 2} we define the framework of point processes, and give a few examples
of such objects, and in Section \ref{sec: 3} we introduce
	single- and multitype branching processes with immigration.
	 	Section \ref{sec: 4} contains results for the %Laplace	transform 
    probability generating functions of the branching process with immigration  spanned by a general point process, and asymptotic results in the case
where the immigration is generated by a
	 Determinantal Point Process (DPP). More precisely, in Theorem \ref{thm: uni} the generating functions of  branching processes with immigration  spanned by DPPs, Cox process and Fractional Poisson Process (FPP) are derived in the single-type case, while Theorem \ref{thm: multi} presents the multivariate analogue.
     In Subsection \ref{subsec: 4.3}, for the case of DPP immigration, we establish equations for the mean and covariances of the process. Furthermore, in Theorem \ref{thm: DPP}, we prove limit results analogous to those in the case without immigration, covering sub-, super-, and critical cases, respectively.      
	  Section \ref{sec: proofs} contains the proofs of the new results.  

\section{Point processes on the real line}
\label{sec: 2}
As we will be interested in point processes (PPs)
describing the times of immigrants joining a branching process, we define the basic notions for PPs on $\mathbb{R}_+:=[0,\infty)$ (or just PPs
from now on) although the formalism for defining them on an arbitrary Polish space is similar. 

Informally, we can consider the point process as a random collection of points; however, to formalise
this idea, the basic construction is via
the so-called random measures. 

We follow the presentation of
\cite[Chapter 2.1]{Last-Penrose-18}, however,
another suitable reference is
the excellent monograph by Baccelli et al.
\cite{Baccelli-Blaszczyszyn-Karray-24}.

Let $(\Omega, \mathcal{F}, \P)$ be a probability space, $\mathcal{B}(\R)$ be the usual Borel $\sigma$-algebra on $\R$, and $\mathbf{N}_{<\infty}$ be the set of measures
$\widetilde
{\Phi}$ on $\R_+$ such that for each
$B\in \mathcal{B}(\R_+), \widetilde{\Phi}(B) \in \N_0:=\N\cup \{0\}$.
Further, denote by $\mathbf{N}$ the
set of measures which can be represented as a countable sum of
elements of $\mathbf{N}_{<\infty}$,
and let 
$\mathcal{N}$ be 
the $\sigma$-algebra generated by the sets
\[
\{\Phi \in \mathbf{N}:\Phi(B) = k
\text{ for some }B \in \mathcal{B}(\Rb_+)
\text{ and } k \in \N_0\}.
\]
We call $\Phi$ a random measure or a PP if
it is a random element of
$(\mathbf{N}, \mathcal{N})$, that is,
a measurable mapping $\Phi:
\Omega \to \mathbf{N}$. In this work,
we will consider PPs which are \emph{proper}, i.e., such that there exist
random variables $\kappa,X_1, X_2, \dots$
such that
\[
\Phi = \sum_{i \leq \kappa} \delta_{X_i},
\]
where $\delta_x$ is the Dirac mass at $x$, so $\Phi$ places unit mass at the random locations $X_1, X_2,\dots,X_\kappa$.
 For a deeper and more general presentation, see \cite[Chapter 2.1]{Last-Penrose-18},
\cite[Chapter 1]{Baccelli-Blaszczyszyn-Karray-24}
or \cite[Chapter 2]{Kallenberg-17}.

Some of the characteristics which help
describing a PP include:
\begin{itemize}
	\item its \emph{intensity measure} $\Lambda$, defined
	by $\Lambda(B) = \Ebb{\Phi(B)}$, where 
	$B \in \mathcal{B}(\R_+)$;
	\item its \emph{Laplace functional},
	which characterises the process completely,
	for all test functions $f$,
	\begin{equation}
		\label{def: Laplace}
		\Lc_\Phi(f) := \Ebb{e^{-\int f\D \Phi}} = 
		\Ebb{e^{-\sum_{i \leq \kappa} f(X_i)}}.
	\end{equation}
	The set of test functions can be all positive measurable ones, like in the case of 
	a Poisson process, or these of 
	compact support, like in the 
	case of determinantal point processes;
	\item its \emph{joint intensities} $\rho_k$, if they exist, defined as $\rho_k: \R_+^k \to\R_+$ such that for each disjoint $B_1, \dots, B_k \in \mathcal{B}(\R_+)$,
	\[
	\Ebb{\prod_{i=1}^k \Phi(B_i)} = \int_{B_1 \times \dots \times B_k}  \rho_k (x_1, \dots, x_k) \Lambda(\D x_1) \dots \Lambda(\D x_k).\]

\end{itemize}

\subsection{Poisson processes}
Probably the most used point process, due to its
mathematical tractability, is the Poisson one, which is characterised by the property that its
intensity measure $\Lambda$ is such that:
\begin{enumerate}
	\item for every $B \in \mathcal{B}(\R_+)$,
	$\Phi(B)\sim Pois(\Lambda(B))$ and
	\item for every disjoint $B_1, \dots,
	B_m \in \mathcal{B}(\R_+)$, 
	$\Phi(B_1), \dots, \Phi(B_m)$ are independent.
\end{enumerate}
It is a consequence that the Laplace functional of this PP is then
\begin{equation}
    \label{eq: Laplace PP}
\mathcal{L}_\Phi(f) = \exp\left(-\int_{\R_+} \left(1 - e^{-f(x)}\right)\Lambda(\D x)\right).
\end{equation}
If $\Lambda(\D x) = \lambda \D x$, then we say that the Poisson process is \textit{homogeneous} 
of rate $\lambda$.

In the following three subsections, we list
some PPs which are often used and
cover, respectively, repulsive, clustering, and heavy-tailed behaviour of the interarrival times.

\subsection{Determinantal point processes (DPPs)}

Determinantal point processes (DPPs) were introduced by Macchi \cite{Macchi-75} under the name \emph{fermion process}, due to their repulsive behaviour.  Since then they have arisen in various contexts including random matrix theory, zeros of random analytic functions, statistical mechanics, and even machine learning \cite{Kulesza-Taskar-12}.  We refer the reader to \cite[Chapter 5]{Baccelli-Blaszczyszyn-Karray-24} for an introduction of their formalism, and to the survey \cite{Decreusefond-16} for an overview.

We call a point process \(\Phi\) on \(\R_{+}\)
  \((\Lambda,K)\)\emph{-determinantal}
  if $\Lambda$ is a locally finite measure on $\R_+$, and its joint intensities satisfy, for
$\boldsymbol{x} = (x_1, \dots,x_n)$,
\[
\rho_{n}(x_{1},\dots,x_{n})
=\det\lbrb{K(x_{i},x_{j})}_{1\le i,j\le n}=:D(\boldsymbol{x}),
\]
where \(K\colon\R_{+}^{2}\to\R_+\) is symmetric; for $\Lambda^n$-almost all $\boldsymbol{x}$, 
$\lbrb{K(x_{i},x_{j})}_{1\le i,j\le n}$ is non-negative definite, and for each bounded $D \in \mathcal{B}(\R_+)$,
$\int_D K(x,x) \Lambda(\D x)$  is finite. These properties ensure that 
$\Lambda$ and $K$ determine uniquely $\Phi$, see
\cite[Corollary 5.1.14]{Baccelli-Blaszczyszyn-Karray-24}. However, they  may look too implicit, as they do not describe how to build admissible kernels. One recipe is to look at
\textit{regular} kernels of the type
$K(x,y) = \sum_{n\in\N}\lambda_n\phi_n(x)\phi_n(y)$ for $\phi_i $ orthonormal
in $L^2(\Lambda, \R_+)$ and $\lambda_i\in[0,1]$,
see \cite[Theorem 5.2.5]{Baccelli-Blaszczyszyn-Karray-24}.

Further, the Laplace functional of \(\Phi\) is given, for any nonnegative \(f\) of compact support, and
\[
\varphi_n(\boldsymbol{x}) := 
\prod_{i=1}^n (1 - e^{-f(x_i))}),
\quad
\quad
\Lambda(\D \boldsymbol{x}) := \Lambda(\D x_1)\dots \Lambda(\D x_n)
\]
by
\begin{equation}\label{eq: DPP Laplace}
\begin{split}
\mathcal{L}_\Phi(f) &= 1 + \sum_{n \geq 1}
\frac{(-1)^n}{n!}
\int_{\R_+^n}\rho_n(x_1, \dots , x_n)
\varphi_n(x_1, \dots, x_n)\Lambda(\D x_1)\dots \Lambda(\D x_n)\\
&=1 + \sum_{n \geq 1}
\frac{(-1)^n}{n!}
\int_{\R_+
^n}D(\boldsymbol{x})\varphi_n(\boldsymbol{x})\Lambda(\D \boldsymbol{x}),
\end{split}
\end{equation}
see for example \cite[Proposition 5.1.18 and Corollary 5.1.19]{Baccelli-Blaszczyszyn-Karray-24} or \cite[Theorem 3.6 with \(\alpha=-1\)]{Shirai-Takahashi-2003}. In particular, note that if $\Lambda (\D x)$ is diffuse and
we choose the kernel $K(x,y)=\ind{x=y}$, then  for all $n$, $\rho_n$ is equal to 1 $\Lambda(\D \boldsymbol{x})$-almost everywhere, so we
obtain the usual $\Lambda$-Poisson process, see
\cite[Example 5.1.6]{Baccelli-Blaszczyszyn-Karray-24} for a rigorous derivation. 

\subsection{Cox Processes}
Cox processes, also known as doubly stochastic Poisson processes, extend the ordinary Poisson process by randomising its intensity measure.  
They are usually used for modelling phenomena where event occurrence is influenced by underlying random factors, as demonstrated in \cite{ShotNoise, Syversveen-98, Yannaros-88} and more recently in \cite{MR3161587}, which discuss their applications in spatial and spatio-temporal data analysis.

Formally, let $\eta$ be a random $\sigma$-finite measure on $\R_+$. Then we call a point process $\Phi$ a Cox process with  \emph{directing measure} $\eta$ if conditional on $\eta$
\[
\Phi \mid \eta \sim 
\text{$\eta$-Poisson process}.
\]
Therefore, its Laplace functional is obtained by averaging over $\eta$
in \eqref{eq: Laplace PP}
\begin{equation}
    \label{eq: Laplace Cox}\Lc_\Phi(f)
=\Ebb{\exp\Bigl(-\int_{\R_+}\bigl(1-e^{-f(x)}\bigr)\,\eta(\D x)\Bigr)},
\end{equation}
for every nonnegative measurable $f$.
The first two moment measures follow by conditioning on $\eta$ and using the Poisson moment formula, as in \cite[Proposition 13.6]{Last-Penrose-18}:
\[
\Ebb{\Phi(B)} = \Ebb{\eta(B)},
\qquad
\Var\bigl(\Phi(B)\bigr)
= \Var\bigl(\eta(B)\bigr) \;+\;\Ebb{\eta(B)},
\]
for all $B\in\mathcal{B}(\R_+)$. For proofs and more properties, see for example \cite[Chapter 13]{Last-Penrose-18}.

\subsection{Fractional Poisson processes}
The fractional Poisson process (FPP), introduced by Mainardi et al. \cite{Mainardi-Gorenflo-Scalas-04} is a non-Markovian  generalisation of the standard homogeneous Poisson process, which has heavy-tailed interarrival times; see also the monograph by Meerschaert and Sikorskii \cite{Meerschaert-Sikorskii-19}. It
depends on parameters $\beta \in(0,1]$ and $\lambda>0$, and is of renewal type,
that is, the atoms of the associated random measure $\Phi_{\beta, \lambda}$ are situated
at $X_n = T_1 + \dots + T_n$ with iid
$T_i$ such that
\[
\P(T_i>t) = E_\beta(- \lambda t^\beta),
\quad
\text{where}
\quad
E_\beta(z) := \sum_{k=0}^\infty
\frac{z^k}{\Gamma(1+\beta k)}
\]
is
the Mittag-Leffler function.
    Define the counting process by
$N_{\beta,\lambda}(t) := \Phi_{\beta,\lambda}((0,t])$.
For $\beta=1$, we get $T_i \sim Exp(\lambda)$, so
the FPP is in fact a homogeneous Poisson process of rate $\lambda$. However, for $\beta\in(0,1)$,
$\P(T_i>t)$ is of
order $1/t^\beta$ and a fortiori
the expectation
$\Ebb{T_i}$ is infinite. 
The work of Meerschaert et al. \cite{Meerschaert-11} extends the link to  Poisson processes beyond $\beta=1$ by showing that 
\begin{equation}
    \label{eq: timechange}
N_{\beta,\lambda}(t) = \mathcal{N}_{\lambda}(Y_\beta(t)),
\end{equation}
where $\mathcal{N}_{\lambda}$ is a homogenous Poisson process of rate $\lambda > 0$, and $Y_\beta$ is the inverse of an independent $\beta$-stable subordinator.
Thus understanding the asymptotic properties of inverse subordinators  \(Y_\beta\), for example as in \cite{Ascione-Savov-Toaldo-24}, is key to deriving limit theorems and other results for the FPP.
Consequently,
the Laplace functional of
$\Phi_{\beta, \lambda}$ can be written
\begin{equation}
    \label{eq: Laplace FPP}
\mathcal{L}_{\Phi_{\beta, \lambda}}(f) = \Ebb{
\exp\lbrb{-\lambda \int_{\Rb_+} \lbrb{1-e^{
-f(t)}}\D Y_\beta(t)}}.
\end{equation}

\section{Branching processes with immigration}
\label{sec: 3}
In this section we present the branching 
process with immigration and fix the
relevant notation.

\subsection{Single-type
	processes}
We denote by $Z(t)$ the number of
particles at time $t$ and outline that
we use $G_X$ for the generating
function of a random variable $X$.

The underlying branching process 
without immigration $Z_\times$ is
a Markov branching process in continuous time with starting
state $Z_{\times}(0)= I$, where $I$ will be the law of the size of a single immigrant group. Therefore we can construct the process with immigration $Z$ as follows:
\begin{itemize}
	
	\item At times $T_1 < T_2< \dots$,
	new particles (immigrants) $I_1, I_2, \dots$ arrive with $I_k\sim I$ iid. We assume that $\Ebb{I}$ is finite.
	
	\item Each particle evolves independently of the others and
	at $Exp(1/\mu)$ time
	dies and produces $\nu$ new particles. It is then known, see
	\cite[(5) on p.106]{Athreya-Ney-72},  that
	\[
	\frac{\partial}{\partial t} G_{Z_\times}(t, s) = f_\nu(G_{Z_\times}(t,s)),
	\]
	where $G_{Z_\times}(t,s):=G_{Z_\times(t)}(s)$ and $f_\nu(s) = (G_\nu(s)-s)/\mu$.
	Also,
	$M_\times(t) := \Ebb{Z_\times(t)} = e^{\rho t}$
	with  $\rho :=f_\nu'(1)= (\Ebb{\nu}-1)/\mu$. 
	The process is called subcritical, 
	critical or supercritical according
	to whether $\Ebb{\nu}$ is, respectively, less, equal, or larger than 1.
	
\end{itemize}
Therefore, the process with immigration can be represented as
\[
Z(t) = \sum_{i:T_i\leq t} {Z}^{(i)}_{\times}(t-T_i),
\]
where ${Z}^{(i)}_{\times}$ are
iid copies of the process
without immigration
$Z_\times$ started with ${Z}^{(i)}_{\times}(0) = I_i$.
\subsection{Multitype processes}
Consider now a generalisation of the
process presented above which is composed of $d$ types of particles. 
We note that we will use bold symbols
$\boldsymbol{\nu}$ for a vector, 
and, unless stated otherwise explicitly, $\nu_{i}$ will denote the value of its $i$th coordinate. We define the generating function in this case as
\[
G_{\boldsymbol{\nu}}(\boldsymbol{s}) := \sum_{\boldsymbol{n}\in \N_0^d}
\P\left(\boldsymbol{\nu}=\boldsymbol{n}\right)\prod_{i=1}^{d}
s_i^{n_{i}}
\quad
\text{for}
\quad
\boldsymbol{s}=(s_1, \dots, s_d).\]
The multitype branching process $\boldsymbol{Z}$ with
$\boldsymbol{Z}(0)=\boldsymbol{I}^{(0)}\sim\boldsymbol{I}$,
is described by:
\begin{itemize}
	\item At times $0 = T_0<T_1<T_2<\dots$, iid
	immigrants
	$\boldsymbol{I}^{(0)}, \boldsymbol{I}^{(1)}, \boldsymbol{I}^{(2)},
	\dots\sim\boldsymbol{I}$ join the population described by the branching process $\boldsymbol{Z}.$
	\item Each particle of type $i$ lives
	$Exp(1/\mu_i)$ time after which it dies
	and produces new ones with distribution $\boldsymbol{\nu}_i$
	of their counts,    that is to say $({\nu}_i)_{j}$ of type $j$.
	
\end{itemize}\label{conditions}
\textbf{Assumptions}: In this work we assume that
\begin{enumerate}
    \item the immigrants and progeny, $I_i$ and $(\nu_i)_j$, have finite expectation, and the latter is not a.s. constant. For some of the results we impose the following stronger restriction:
    \begin{equation}\tag{H}\label{eq:strong}
        I_i \text{ and } (\nu_i)_j \text{ have finite second moments.}
    \end{equation}
    \item the mean offspring matrix
    \[
      M := \bigl(\Ebb{(\nu_i)_j}\bigr)_{1\le i,j\le d}
    \]
    is assumed 
    \emph{primitive}, i.e.\ there exists an integer $p\ge1$ such that $M^p$ has strictly positive entries. We briefly discuss the more general \textit{decomposable} (or also \textit{reducible}) case in Subsection~\ref{sec:discussion}.
\end{enumerate}

%------------------------------------------------------------

Similarly to the single-type case,
\[
\boldsymbol{Z}(t) = \sum_{i:T_i\leq t}
\boldsymbol{Z}^{(i)}_{\times}
(t-T_i),
\]
where $\boldsymbol{Z}^{(i)}_{\times}$ are iid processes with law $\boldsymbol{Z_\times}$, that is,
without immigration and started with $\boldsymbol{I^{(i)}}$.
% Note that this implies that
% \[
% G_{\boldsymbol{Z}_{\times, i}(t)}(\boldsymbol{s}) = G_{\boldsymbol{I}}\lbrb{G_{\boldsymbol{Z}_\times(t)}(\boldsymbol{s})}
% \]
% for a process $\boldsymbol{Z}_{\times}$ with $\boldsymbol{Z}_\times(0)= 1 $
This gives
the generating function of the process,
with
$G_{\boldsymbol{Z}}(t,\boldsymbol{ s}):=
G_{\boldsymbol{Z}(t)}(\boldsymbol{ s})
$ and $G_{\boldsymbol{Z}_{\times}}(t,\boldsymbol{s}):=
    G_{\boldsymbol{Z}_{\times}(t)}(\boldsymbol{s})$,
\begin{equation}
	\label{eq: generating}
	G_{\boldsymbol{Z}}(t,\boldsymbol{ s})=
	\Ebb{\prod_{i:T_i\leq t}G_{\boldsymbol{Z}^{(i)}_{\times}}(t-T_i,\boldsymbol{s})}
	= \Ebb{\exp\lbrb{\sum_{i:T_i \leq t}\ln\lbrb{G_{\boldsymbol{Z}_{\times}}(t-T_i,\boldsymbol{s})}}}.
\end{equation}
Translating the last in the language
of the Laplace functional of the random measure
$\Phi$
governing the process $T_i$, substituting in \eqref{def: Laplace}, we get that
\begin{equation}
    \label{eq: generating mu}
G_{\boldsymbol{Z}}(t,\boldsymbol{s})
= \mathcal{L}_\Phi(f_t),
\quad
\text{with}
\quad
f_t(x):=-\ln\lbrb{G_{\boldsymbol{Z}_{\times}}\lbrb{t-x,\boldsymbol{s}}}\ind{x\leq t}.
\end{equation}
If we want to work instead with the
Laplace transform,
\[
\Lc_{\boldsymbol{Z}}(t,\boldsymbol{s}):=\Ebb{\exp(-\langle\boldsymbol{Z}(t),\boldsymbol{s}\rangle)},
\]
where $\langle\cdot,\cdot\rangle$ is usual scalar product on $\R^d$,
we would have
\begin{equation}
	\label{eq: Laplace mu}
	\mathcal{L}_{\boldsymbol{Z}}(t,\boldsymbol{s})
	= \Lc_\Phi(g_t),
	\quad
	\text{with}
	\quad
	g_t(x):=-\ln\lbrb{\Lc_{\boldsymbol{Z}_{\times}}\lbrb{t-x,\boldsymbol{s}}}\ind{x\leq t}.
\end{equation}
\subsection{Decomposable (Reducible) Case}\label{sec:discussion}

In the more general \emph{decomposable} (or \emph{reducible}) setting, the mean offspring matrix~$M$ can be permuted into a block-triangular form
\[
  P M P^{-1} = \begin{pmatrix}
    M_{11} & M_{12} & \cdots & M_{1k} \\
    0      & M_{22} & \cdots & M_{2k} \\
    \vdots & \vdots & \ddots & \vdots \\
    0      & 0      & \cdots & M_{kk}
  \end{pmatrix},
\]
where each diagonal block~$M_{ii}$ is irreducible (primitive).  One then studies the multitype branching process in stages:

\begin{itemize}
  \item \textbf{Block~1 (leading class):} behaves as an irreducible system with its own growth rate; immigration into this class follows the point process mechanism as before.
  \item \textbf{Block~$r$ (subsequent class):} receives external immigration from class~$1,\dots,r-1$ via their off-diagonal connections $M_{jr}$, in addition to the original point-process arrivals.  One treats the output of earlier blocks as an inhomogeneous immigration input and applies the irreducible theory to each block in turn.
\end{itemize}

Asymptotic results for each class can then be obtained by iterating the single-block analysis, noting that the effective immigration intensity into block~$r$ is a superposition of point-process-driven arrivals plus the evolving contribution from blocks~$1$ through $r-1$.  

\section{Asymptotic behaviour
	at infinity
	of the branching process with immigration }
\label{sec: 4}

\subsection{Single-type processes}
The continuous-time Markov branching process without immigration is presented in Athreya and Ney \cite[Chapter III]{Athreya-Ney-72}. We refer again to Rahimov \cite{Rahimov-21}
for an extensive recent review of branching processes with immigration and to Vatutin and Zubkov \cite{Vatutin-Zubkov-87}
for
a review of the classical results until 1983.

A standard argument using martingale theory,
for example \cite[Theorem 1 on p.111]{Athreya-Ney-72}, shows that for the process 
without immigration $Z_{\times}$, there exists a real-valued random variable $W_\times$ such that
\begin{equation}
	\label{eq: limit}
	\frac{Z_\times(t)}{e^{\rho t}} \xrightarrow[t \to \infty]{a.s.} W_\times.
\end{equation}

In the supercritical case,
it is natural to characterise some of the properties of $W_\times$, e.g., the existence of a
density.
In the sub- and critical cases, this $W_\times$ is a.s. equal to $0$ and a common question is
to analyse the so-called Yaglom limit 
as $t\to\infty$
of quantities such as $\P(Z_\times(t)>f(t)|Z_\times(t)>0)$ for some $f$,
see \cite[Chapter III, Theorems 2-4]{Athreya-Ney-72}. Similar results are also available for the case of age-dependent processes, see \cite[Chapter IV]{Athreya-Ney-72}.

When an immigration component is added, it is often possible to extract asymptotic results for $Z$ using the information for the underlying process $Z_\times$ through representation
\eqref{eq: generating}. The case of Poisson immigration is considered when
the underlying process is subcritical in \cite{Hyrien-Mitov-Yanev-15-sub}, critical in \cite{Mitov-Yanev-13, Mitov-Yanev-24,Bojkova-Yanev-19-crit}, and supercritical in \cite{Hyrien-Mitov-Yanev-13-super}.

We provide the generating functions in the cases of DPPs, Cox, and FPPs. After choosing
an exact model, this can provide the limit of a scaled $Z$ through the results for $Z_\times$. We also note that it is possible to work with the generating function of the process
starting with exactly one particle $Z_\times^{(1)}$, that
is
$Z_\times^{(1)}(0) = 1$, via the relation
\begin{equation}
	\label{eq: functional 1d}
	G_{Z_{\times}}(s) = G_I\lbrb{G_{Z_\times^{(1)}}(s)}.
\end{equation}

\begin{theorem}\label{thm: uni}
	\begin{enumerate}
			\item In the case of immigration spanned by a $(\Lambda, K)$-DPP process, the expectation		$\Ebb{s^{Z(t)}}$ is equal to
		\begin{equation}
		    \label{eq: Laplace DPP}
		1 + \sum_{n\geq 1}\frac{(-1)^{n}}{n!}\int_{(0,t]^n}
		D(\boldsymbol{x})
		\prod_{i=1}^n\lbrb{1 - G_{Z_{\times}}(t-x_i,s)} \Lambda(\D \boldsymbol{x})
				\end{equation}
		\item In the case of immigration spanned by a Cox process with a
		directing measure $\eta$, 
		\[
		\Ebb{s^{Z(t)}} = {\mathbb{E}\lbb \exp\lbrb{-
				\int_{(0,t]}\lbrb{1-G_{Z_{\times}}(t-x,s)}\eta(\D x)}\rbb}.
		\]
		\item 
	 In the case of immigration spanned by a
	 $(\beta,\lambda)$-FPP, 
	\[
	\Ebb{s^{Z(t)}} = {\mathbb{E}\lbb \exp\lbrb{- \lambda
			\int_{(0, t]}\lbrb{1-G_{Z_{\times}}\lbrb{t-x,s)}} \D Y_\beta(x)}\rbb},
	\]
    where $Y_\beta$
    is the inverse of an independent
    $\beta$-stable subordinator.
	\end{enumerate}
To obtain the respective Laplace transforms, one should
replace $G_{Z_\times}$ with $\mathcal{L}_{Z_\times}$.
\end{theorem}
\begin{proof}
	The generating functions are directly obtained by substituting
    \eqref{eq: generating mu} into the relevant Laplace functional, that is, 
    \eqref{eq: Laplace Cox} for Cox processes, \eqref{eq: Laplace FPP}
    for FPPs, and \eqref{eq: DPP Laplace} for DPPs.
For the Laplace transform, instead of \eqref{eq: generating mu}, we should substitute \eqref{eq: Laplace mu}.
\end{proof}

To make a connection with the previously presented classes of point processes,
the second item of the last theorem was proved by Butkovsky \cite{Butkovsky-12} and applied in the context of branching processes.

As for the fractional point process, since it is a renewal process, we can apply the results of
Kaplan and Pakes \cite{Kaplan-Pakes-74-super, Kaplan-Pakes-74-sub} for the sub- and supercritical cases, see also \cite{Pakes-Parthasarathy-75}. The only specificity one needs to take care of is that the interarrival times have infinite expectation, however the techniques are often the same, see \cite[Remarks on page 379, 385, and 389]{Kaplan-Pakes-74-super}.

% \subsubsection{Moments}
\subsection{Multitype processes} The results for multitype
processes have similar nature, however generalising the one-dimensional
results is not always direct and requires a careful treatment. 
To extend $\eqref{eq: limit}$, let us define an analogue of $\Ebb{\nu}$, that is $e^A$, where the generator  $d \times d$ matrix $A$ is defined by 
\begin{equation}
	\label{eq:def matrix A}
A_{i,j} := \mu_i^{-1} \lbrb{\Ebb{\lbrb{\nu_i}_j} - \ind{i= j}}.
\end{equation}
We will suppose that $A$ is irreducible and let $\rho$ be its largest eigenvalue (also known as Perron-Frobenius root).
The process is called sub-, super- or just critical according to $\rho$, respectively,
less, larger or equal than 0. Let $\boldsymbol{u}$ and $\boldsymbol{v}$ be 
row vectors such that
\[A\boldsymbol{u}^t = \rho \boldsymbol{u}^t,
\quad
 \boldsymbol{v}A = \rho \boldsymbol{v},
 \quad
 \sum_i u_iv_i = 1,
 \quad
 \text{and}
 \quad \sum_i u_i = \\1,\]
 with \emph{$t$} denoting the transpose. Moreover,
 it is known, see
 \cite[p.203]{Athreya-Ney-72} that the coordinates of $\boldsymbol{u}$ 
 and
 $\boldsymbol{v}$
 are strictly positive. Then, from
\cite[Theorem 2, p.206]{Athreya-Ney-72}, rewritten in the form \cite[Lemma 2.1]{Rabehasaina-21}, there exists a real-valued random variable $W_\times$ such that
\begin{equation}
	\label{eq: limit ndim}
	\frac{\boldsymbol{Z}_{\times}(t)}{e^{\rho t}} \xrightarrow[t \to \infty]{a.s.}\boldsymbol{v} \cdot W_\times ,
\end{equation}
with $\cdot$  the usual multiplication, spelled out for better readability. We recall that $\boldsymbol{Z}_{\times}(t)$ is  the process without immigration started with random number of particles according to the law $\boldsymbol{I}$. 
For completeness, we also note that the equivalent of the functional equation \eqref{eq: functional 1d} is
\[
G_{\boldsymbol{Z}_{\times}}(\boldsymbol{s}  ) = G_{\boldsymbol{I}}\lbrb{G_{\boldsymbol{Z}^{(1)}_{\times}}(\boldsymbol{s}),
	\dots,
	G_{\boldsymbol{Z}^{(d)}_{\times}}(\boldsymbol{s})},
\]
with ${\boldsymbol{Z}^{(k)}_{\times}}$ a process without immigration
started with exactly one particle of type $k$.

Furthermore, a convenient fact is that Theorem \ref{thm: uni} holds when changing
everything to its multidimensional
equivalent (which we write in \textbf{bold}).

\begin{theorem}\label{thm: multi}
The statements of Theorem~\ref{thm: uni} extend  to the vector case by $s\to\boldsymbol{ s}$, $Z\to\boldsymbol Z$, in all regimes (subcritical, critical, supercritical).  The same holds if one replaces generating functions by Laplace transforms $G_{\boldsymbol Z}\to\mathcal L_{\boldsymbol Z}$ and $G_{\boldsymbol Z_\times}\to\mathcal L_{\boldsymbol Z_\times}$.
\end{theorem}

Therefore, in theory, we can use results for the process without immigration $\boldsymbol{Z}_\times$
to obtain information about $\boldsymbol{Z}$.

As noted previously, the case of immigration generated by a Poisson random measure is
analysed   'in \cite{Mitov-Yanev-Hyrien-18-critical-multi,  Bojkova-Hyrien-Yanev-22, Bojkova-Hyrien-Yanev-23} and independently by Rabehasaina and Woo in \cite{Rabehasaina-21}. For
immigration times that form a renewal process, see
\cite{Slavtchova-Bojkova-96}. The latter may be applied in the case of an immigration
spanned by a FPP after checking that the used arguments can be modified to include the case of an infinite expectation
of the inter-arrival times.

As an illustrative example of how ideas from the case of Poisson immigration can be modified to
encompass the case of immigration spanned by a DPP, we generalise
\cite[Theorems 3, 4, and 5]{Rabehasaina-21} (since the Poisson point process is a DPP),
which are also available in \cite{Mitov-Yanev-Hyrien-18-critical-multi,  Bojkova-Hyrien-Yanev-22, Bojkova-Hyrien-Yanev-23}.
\section{Multitype branching processes with DPP immigration}
\subsection{Moments of the process}
\label{subsec: 4.3}
Let us recall that
for a vector
$\boldsymbol{Z}$,
we denote by $Z_i$ its $i$th coordinate. Then we have the following.\begin{theorem}\label{thm: moments}
	In the case of immigration spanned by a $(\Lambda, K)$-DPP,	\begin{enumerate}
		\item  we have that
		\[\Ebb{{Z}_i(t)} = \int_{(0,t]}
		K(x,x)\Ebb{Z_{\times,i}(t-x)}\Lambda(\D x),\]
		\item and if we assume \eqref{eq:strong},
        \begin{align*}
			Cov(Z_i(t), Z_j(t))=&
			\int_{(0,t]} K(x,x)
		 \Ebb{Z_{\times, i}(t-x) Z_{\times, j}(t-x)} \Lambda(\D x) \\
& \! - \int_{(0,t]^2} K^2(x,y) \Ebb{Z_{\times, i}(t-x)} \Ebb{Z_{\times, j}(t-y)} \Lambda(\D x) \Lambda(\D y).
		\end{align*}
	\end{enumerate}
	
\end{theorem}
\begin{proof}
This follows by differentiating the generating function of the process, derived as the multidimensional version of \eqref{eq: Laplace DPP}. We provide the complete calculation in Section \ref{sec: proof moments}.
	
\end{proof}
\subsection{Asymptotics}
\begin{theorem} \label{thm: DPP}
	Let $\rho$ be the Perron-Frobenius root of the matrix $A$. In the case of immigration spanned by a $(\Lambda, K)$-DPP $\Phi$ with kernel
	$K$ such that
	$\int_{(0,\infty)}e^{-\rho x}K(x,x)\Lambda(\D x)$
    is finite,
	there exists an $\R^d$-valued random variable $\boldsymbol{W}$
	such that
	\begin{equation}\label{eq: Laplace DPP sub}
	\frac{\boldsymbol{Z}(t)}{e^{\rho t}} \xrightarrow[t \to \infty]{d} \boldsymbol{W}.
	\end{equation}
	Moreover,
	\[
	\boldsymbol{W}\stackrel{d}{=}
	\sum_i
	\boldsymbol{v}\cdot W^{(i)}_{\times}e^{-\rho T_i},
    \quad
    \text{and}
    \quad \Ebb{\exp\lbrb{-\langle\boldsymbol{W},\boldsymbol{s}\rangle}}=
    \mathcal{L}_{\Phi}
	\lbrb{-\ln\lbrb{\mathcal{L}_{\boldsymbol{v}W_\times
			}\lbrb{\boldsymbol{s}e^{-\rho x}}}},
	\]
	where $W^{(i)}_{\times}$ are  iid
	copies of $W_{\times}$
	and $T_i$ being the
	atoms of $\Phi$.
\end{theorem}
The proof of the last theorem is presented in Section \ref{sec: proof DPP1}.

For the next results, to introduce some regularity of the DPP, we assume it is \emph{stationary}, meaning that for any \( t \in \mathbb{R} \) and \( \boldsymbol{x} \in \mathbb{R}^n \),
\begin{equation}\label{eq: stationary}
	D(\boldsymbol{x}+t) = D(\boldsymbol{x}),
\end{equation}
where \( \boldsymbol{x} + t := (x_1 + t, \dots, x_n + t) \); see \cite[Chapter 5.6]{Baccelli-Blaszczyszyn-Karray-24}. This property holds, for example, if the kernel is shift-invariant, \( K(x, y) = K(x + t, y + t) =: C(x - y) \), as in the case of the Poisson kernel, \( K(x, y) = \ind{x=y} \), or the Ginibre (Gaussian) kernel: \( K(x, y) = \exp\left( -(x-y)^2/2 \right) / \pi \). A result by Lavancier et al. \cite[Theorem 1]{Lavancier-15} states that \( C \) can be any continuous correlation function with eigenvalues in its spectral representation bounded by 1, with examples provided in their work.

% \begin{proposition}
	% \begin{enumerate}
		%     \item We have that $\Ebb{Z_i(t)} \sim $
		% \end{enumerate}
	
	% \end{proposition}
    
Before stating our next result, we introduce the asymptotic equivalence notation, provided
$g(t)$ is non-zero for sufficiently large $t$,
    \[
    f(t)\simi g(t)
    \quad
    \text{if}
    \quad
    \lim_{t\to\infty}\frac{f(t)}{g(t)} = 1.
    \]
\begin{theorem}\label{thm: DPP2}
	Assume that immigration 
	is spanned by a
	stationary $(\Lambda,K)$-DPP
	with $\Lambda(\D x) = \lambda(x) \D x$,
	and that $\lambda(x)\simi\lambda_\infty e^{\delta x}$ for some $\lambda_\infty > 0$ and $\delta \in \R$.

	\begin{enumerate}
		\item \label{it:1} Under \eqref{eq:strong}, 
		if $\delta> \max\{\rho,0\}$, 
		then 
		\[
		\frac{{Z_i}(t)}{e^{\delta t}}\xrightarrow[t \to \infty]{L^2} A_i:=
 K_\ast \lambda_\infty\int_0^\infty e^{-\delta x}\Ebb{Z_{\times,i}(x)}\D x,
		\]
        where $K_\ast:=K(0,0)$.
		\item \label{it:2} Under \eqref{eq:strong}, 
		if the process is \textbf{supercritical}, that is, 
		$
		\rho>0$, and  $\delta = \rho$, then
		\[\frac{{Z_i}(t)}{te^{\delta t}}
		\xrightarrow[t\to\infty]{L^2} A_i':=K_\ast \lambda_\infty 
        \langle \boldsymbol{u},
	\Ebb{\boldsymbol{I}}\rangle v_i.\]

		\item  
		\label{it:3} 
		If the process is \textbf{critical}, that is		$
		\rho=0$, $\Ebb{||\boldsymbol{\nu}||^2}$ is finite, $\lambda$ is bounded, $\delta = \rho$, and the covariance function $K(x, 0)$ of the DPP tends to $0$ as $x\to \infty$. Then
		\[\frac{\boldsymbol{Z}(t) }{t}\xrightarrow[t\to\infty]{d} 
		Y
		\boldsymbol{v}  ,\]  where $Y\sim \Gamma(K(0,0)\lambda_\infty \beta,1/Q)$ with
\[
Q := \frac{1}{2} \sum_{i, j, k = 1}^d \left. \frac{\partial^2 G_{\nu_i}}{\partial x_j \partial x_k} \right|_{\boldsymbol{x} = \boldsymbol{1}} \mu_i^{-1}v_i u_j u_k,
    \qquad
	\beta :=
	\frac{\langle\boldsymbol{u}, \Ebb{\boldsymbol{I}}\rangle}{Q},
	%  \sum_{i=1}^d
	% \Ebb{\boldsymbol{I}_j}{\boldsymbol{u}_j}
\]
    and $\Gamma(\alpha, \beta)$ is
		the Gamma law with shape $\alpha$
		and rate $\beta$, that is,
		its density is, for $x >0$,
		\[
		f_{\alpha,\beta}(x) = 
		\frac{\beta^\alpha}{\Gamma(\alpha)}x^{\alpha-1}e^{-\beta x}.
		\]
		\item \label{it:4}
		If the process is \textbf{subcritical}, that is, 
		$
		\rho<0$, and $\delta = 0$, then
		\[
		\boldsymbol{Z}(t) \xrightarrow[t\to\infty]{d} \boldsymbol{X},
		\]where $\boldsymbol{X}$ is a random variable whose Laplace transform is available in
		\eqref{eq:Laplace limit}.

	\end{enumerate}
\end{theorem}

\begin{remark}
	\label{rem: eigen} The result of item \ref{it:3} in the last theorem in the case of Poisson immigration is established  
	in \cite[Theorem 8]{Mitov-Yanev-Hyrien-18-critical-multi} and, under
the stronger assumption that all
	moments of $I_i$ exist, in 
	\cite[Theorem 3]{Rabehasaina-21}.  After a close inspection, we have observed that the two in fact do not match. This is due to
two errata: in the proof of \cite[Theorem 8]{Mitov-Yanev-Hyrien-18-critical-multi}, there is a missed $r^\ast$ on the second line of page 221. The error in \cite[Theorem 3]{Rabehasaina-21} is due to an incorrect
application to the result of Weiner \cite{Weiner-70}.
	In his work, the eigenvectors $\boldsymbol{u}_W$ 
	and $\boldsymbol{v}_W$ are defined as, respectively, the right and left eigenvector of the matrix composed of $\Ebb{(\nu_{i})_j}-\ind{i=j}$, and not as the eigenvectors of the generator
	matrix $A$, see \eqref{eq:def matrix A}. The normalisation is
 $\langle \boldsymbol{u}_W, \boldsymbol{1}\rangle = 1$ and 
 $\langle \boldsymbol{u}_W, \boldsymbol{v}_W\rangle = 1$.
	It is then direct that,
	\begin{equation}
		\label{eq: transfer}
  u_{W,i}:=(u_{W})_i = u_i, 
  \quad \!
  v_{W,i} = \frac{1}{\sum_j (u_j v_j/ \mu_j)}
  \frac{v_i}{\mu_i},
  \quad \!
  \text{and}
  \quad\!
  \sum_i \mu_i u_{W,i} v_{W,i} = \frac{1}{\sum_i (u_i v_i/\mu_i)}. 
	\end{equation}
	The definition of $\boldsymbol{u}$ and
	$\boldsymbol{v}$ in this work is aligned with 
	these in
	Sevastyanov \cite{Sevastyanov-71} and
	Athreya and Ney
	\cite{Athreya-Ney-72}.
\end{remark}
The theorem is proved in
Section \ref{sec: proof DPP2}. For the proof of the first two items, we will need the following  estimates.
\begin{lemma}\label{lemma 2}
Under the assumptions of 
Theorem \ref{thm: DPP2} and \eqref{eq:strong},
\begin{enumerate}
    \item if $\delta>\rho$, then
for some $A_i \in \Rb_+$,
\[
\Ebb{Z_i(t)}\simi A_ie^{\delta t},
\quad
\text{and}
\quad
\Var(Z_{i}(t)) = 
	\begin{cases}
	\bo{e^{\max\{2\rho, \delta\} t}} & \text{if $\rho > 0$ and $\delta\neq2 \rho$},\\
		\bo{t e^{\delta t}} & \text{if $2\rho=\delta > 0$}.
	\end{cases}
\]
\item if $\delta=\rho>0$, then for some $A_i'\in \Rb_+$,
\[
\Ebb{Z_i(t)}\simi A'_ite^{\delta t},
\quad
\text{and}
\quad
\Var(Z_{i}(t)) = 
		\bo{e^{2\delta t}}.
\]
\end{enumerate}

\end{lemma}

\section{Proofs}
\label{sec: proofs}\subsection{Proof of Theorem \ref{thm: moments}} \label{sec: proof moments}
\begin{proof}
% [Proof of Theorem \ref{thm: moments}]
 We start with the generating function of the process, similarly to \eqref{eq: Laplace DPP}, the generating function of the multitype process,
    \[
    G_{\boldsymbol{Z}}(t, \boldsymbol{s}):=
    \sum_{\boldsymbol{n}\in \N_0^d}
\P\left(\boldsymbol{Z}(t)=\boldsymbol{n}\right)\prod_{i=1}^{d}
s_i^{n_{i}},
    \]
    is equal to
\begin{equation}
 \label{eq: generating DPP}
    1 + \sum_{n\geq 1}\frac{(-1)^{n}}{n!}\int_{(0,t]^n}
		D(\boldsymbol{x})
		\prod_{i=1}^n\lbrb{1 - G_{\boldsymbol{Z}_{\times}}(t-x_i,\boldsymbol{s})} \Lambda(\D\boldsymbol{x}).
\end{equation}
    To obtain the moments, we differentiate with respect to $s_i$ and evaluate at $\boldsymbol{s} = \boldsymbol{1}$. Note that we can interchange $\partial/\partial s_i$ with the integral, because  
     $G_{Z_{\times}}(t,\boldsymbol{s})\in[0,1]$ for each $t$ and $||\boldsymbol{s}||\leq 1$, and the assumption that the expectation of the progeny is finite gives that $\Ebb{Z_{\times,i}(t)}$ is finite for each $t$, see for example \cite[(5) on p.202]{Athreya-Ney-72}. Therefore,
    \[
    \labsrabs{\frac{\partial}{\partial s_i}\prod_{j=1}^n\lbrb{1 - G_{\boldsymbol{Z}_{\times}}(t-x_j,\boldsymbol{s})} }\leq \sum_{i=1}^n \Ebb{Z_{\times, i}(t-x_j)}\leq C_t
    \]
    for some universal $C_t \in \Rb$ as $t$ is fixed. Moreover, as $K$ and $\Lambda$ are locally bounded, we can apply a variation of the dominated convergence theorem, e.g. \cite[Theorem 6.28]{Klenke-2020}. To further exchange the summation and $\partial/\partial s_i$, by
	Hadamard's inequality, see e.g.
	\cite[(15.A.1)]{Baccelli-Blaszczyszyn-Karray-24}, 
	\begin{equation}
		\label{eq:Hadamard}
		D(\boldsymbol{x})
		=
		\det(K(x_i,x_j))_{1\leq i,j \leq n}
		\leq \prod_{i=1}^n K(x_i,x_i),
	\end{equation}
    so
    \[
  \labsrabs{  \int_{(0,t]^n}
		D(\boldsymbol{x})
		\frac{\partial}{\partial s_i}\prod_{i=1}^n\lbrb{1 - G_{\boldsymbol{Z}_{\times}}(t-x_i,\boldsymbol{s})} \Lambda(\D\boldsymbol{x})}
        \leq \lbrb{C_t\int_{(0,t]} K(x,x)
        \Lambda(\D x)}^n =: D_t^n,
    \]
  with $D_t \in \R_+$, and again by dominated convergence, we can interchange $\partial/\partial s_i$ and $\sum_{n\geq1}$ (considering the latter as an integral w.r.t. to a Dirac-type measure). 

  Observe that
    $1-G_{\boldsymbol{Z}_
    \times}(t-x_i, \boldsymbol{1})= 0$, so
    \begin{align*}\
    \Ebb{Z_i(t)} =
    \left.\frac{\partial}{\partial s_i}
     G_{\boldsymbol{Z}}(t, \boldsymbol{s})\right|_{\boldsymbol{s} = \boldsymbol{1}}
     &=- \int_{(0,t]} K(x,x)
     \left.\frac{\partial}{\partial s_i}
     \lbrb{1-G_{\boldsymbol{Z}_\times}(t-x, \boldsymbol{s})}\right|_{\boldsymbol{s}
= \boldsymbol{1}} \Lambda(\D x) 
     \\
     &=
     \int_{(0,t]}
		K(x,x)\Ebb{Z_{\times,i}(t-x)}\Lambda(\D x),
    \end{align*}
    as the derivatives of the terms in \eqref{eq: generating DPP} for $n>1$
    would have a
    multiple $1-G_{\boldsymbol{Z}_
    \times}(t-x_i, \boldsymbol{1})$ for some $i$,
    which is zero when evaluated at $\boldsymbol{s} = \boldsymbol{1}$.
    
    Similarly for the second moments, we can exchange the derivative and sum, because we assume that
    the respective second moments exist, see \eqref{eq:strong}. We are thus interested in
\begin{equation}
    \label{eq:secondm}
    1 + \sum_{n\geq 1}\frac{(-1)^{n}}{n!}\int_{(0,t]^n}
		D(\boldsymbol{x})\lbrb{\left.
		\frac{\partial^2}{\partial s_i\partial s_j}\prod_{i=1}^n\lbrb{1 - G_{\boldsymbol{Z}_{\times}}(t-x_i,\boldsymbol{s})} \right|_{\boldsymbol{s}=\boldsymbol{1}}}\Lambda(\D\boldsymbol{x}).
\end{equation}
but again
    since 
    $1-G_{\boldsymbol{Z}_
    \times}(t-x_i, \boldsymbol{1})= 0$,
    the non-zero terms at
    $\boldsymbol{s}=1$ are obtained for $n=1$ and $n=2$. For $n=1$, we have
    \[
   \left. \frac{\partial^2}{\partial s_i
     \partial s_j} 
     \lbrb{1-G_{\boldsymbol{Z}_\times}(t-x, \boldsymbol{s})}\right|_{\boldsymbol{s}
     = \boldsymbol{1}} = -\Ebb{Z_{\times,i}(t-x)Z_{\times,j}(t-x)},
    \]
    and for $n=2$,
    \begin{align*}
{\frac{\partial}{\partial s_i
}} 
\lbrb{1 - G_{\boldsymbol{Z}_{\times}}(t-x_1,\boldsymbol{s})}
&{\frac{\partial}{\partial s_j
}} 
\lbrb{1 - G_{\boldsymbol{Z}_{\times}}(t-x_2,\boldsymbol{s})}
\\
&+
{\frac{\partial}{\partial s_j
}} 
\lbrb{1 - G_{\boldsymbol{Z}_{\times}}(t-x_1,\boldsymbol{s})}
{\frac{\partial}{\partial s_i
}} 
\lbrb{1 - G_{\boldsymbol{Z}_{\times}}(t-x_2,\boldsymbol{s})},
    \end{align*}
 which evaluated at $\boldsymbol{s}=
 \boldsymbol{1}$ is exactly 
 \[
 F^\ast(i,j,x_1,x_2):=\Ebb{Z_{\times, i}(t-x_1)}\Ebb{ Z_{\times, j}(t-x_2)}
 +
 \Ebb{Z_{\times, j}(t-x_1)}\Ebb{ Z_{\times, i}(t-x_2)}.
 \]
Writing
$F(i,x_k):=\Ebb{Z_{\times, i}(t-x_k)}$ and using that \[\Ebb{{Z}_i(t)} = \int_{(0,t]}
		K(x,x)\Ebb{Z_{\times,i}(t-x)}\Lambda(\D x),\] we have
\begin{align*}
\int_{(0,t]^2}& K(x_1,x_1)K(x_2,x_2)
F^\ast(i,j,x_1,x_2)
\Lambda(\D x_1)\Lambda(\D x_2)
\\
&=
\int_{(0,t]} K(x_1,x_1)F(i,x_1)\Lambda(\D x_1) 
\int_{(0,t]} K(x_2,x_2)F(j,x_2)\Lambda(\D x_2) \\
&\qquad\qquad\qquad
+
\int_{(0,t]} K(x_1,x_1)F(j,x_1)\Lambda(\D x_1) 
\int_{(0,t]} K(x_2,x_2)F(i,x_2)\Lambda(\D x_2) \\
&= 2 \Ebb{Z_i(t)}\Ebb{Z_j(t)}.
\end{align*}
Similarly,
because  by definition $K(x,y) = K(y,x)$, one can also check that
\begin{align*}
    \int_{(0,t]^2} K^2(x_1,x_2)
&F^\ast(i,j,x_1,x_2)\Lambda(\D x_1)\Lambda(\D x_2)\\
  &=2\int_{(0,t]^2} K^2(x_1,x_2)
\Ebb{Z_{\times, i}(t-x_1)}\Ebb{Z_{\times, i}(t-x_2)}\Lambda(\D x_1)\Lambda(\D x_2).
\end{align*} Therefore, substituting in \eqref{eq:secondm}, we get
    \begin{align*}
    &\Ebb{Z_i(t) Z_j(t)} = 
    \left. \frac{\partial^2}{\partial s_i
     \partial s_j}  G_{\boldsymbol{Z}}(t, \boldsymbol{s})\right|_{\boldsymbol{s} = \boldsymbol{1}}
     = -\int_{(0,t]} K(x,x)
     \left.\frac{\partial^2}{\partial s_i
     \partial s_j} 
     \lbrb{1-G_{\boldsymbol{Z}_\times}(t-x, \boldsymbol{s})}\right|_{\boldsymbol{s}
= \boldsymbol{1}} \Lambda(\D x) 
\\&\qquad\quad\,\,\,+\frac{1}{2!}
\int_{(0,t]^2}
		\det(K(x_i,x_j))_{1\leq i,j \leq 2}
	 \left.\frac{\partial^2}{\partial s_i
     \partial s_j} 
      \lbrb{	\prod_{i=1}^2\lbrb{1 - G_{\boldsymbol{Z}_{\times}}(t-x_i,\boldsymbol{s})} }\right|_{\boldsymbol{s}
= \boldsymbol{1}}\Lambda(\D x_1)  \Lambda(
		\D x_2)
\\
&= 	\int_{(0,t]} K(x,x)
\Ebb{Z_{\times, i}(t-x) Z_{\times, j}(t-x)} \Lambda(\D x) \\
& \qquad\quad\,\,\,+\frac{1}{2}\int_{(0,t]^2} \Big(K(x_1,x_1)K(x_2,x_2) -K^2(x_1, x_2)\Big)
F^\ast(i,j,x_1,x_2)
\Lambda(\D x_1) \Lambda(\D x_2)\\
&=\int_{(0,t]} K(x,x)
\Ebb{Z_{\times, i}(t-x) Z_{\times, j}(t-x)} \Lambda(\D x) \\
&\qquad\quad\,\,\,+ \Ebb{Z_i(t)}\Ebb{Z_j(t)}
- \int_{(0,t]^2} K^2(x_1,x_2)
\Ebb{Z_{\times, i}(t-x_1)}\Ebb{Z_{\times, i}(t-x_2)}\Lambda(\D x_1)\Lambda(\D x_2).
    \end{align*}
    Using the covariance formula
    gives exactly the second item
    of Theorem \ref{thm: moments}, which concludes the proof.
\end{proof}
\subsection{Proof of Theorem
	\ref{thm: DPP}}
\label{sec: proof DPP1}
\begin{proof}
% [Proof of Theorem \ref{thm: DPP}]
We will work with the Laplace transform of $\boldsymbol{Z}$,
	\[
	\mathcal{L}_{\boldsymbol{Z}}(t,\boldsymbol{s}) :=
	\Ebb{\exp\lbrb{-\langle \boldsymbol{Z}(t), \boldsymbol{s}\rangle}},
	\]
	and use a similar technique to
 \cite[Theorem 3]{Rabehasaina-21}.
	To facilitate the formalism, define $\boldsymbol{Z}(t)=0$ for $t <0$. Using \eqref{eq: limit ndim}, we have that for positive $x$,
	\[
	\frac{\boldsymbol{Z}_{\times}(t-x)}{e^{\rho t}}
	\xrightarrow[t \to \infty]{a.s.}
	\boldsymbol{v}\cdot e^{-\rho x} W_\times ,
	\]
	which gives the pointwise convergence
	of the Laplace transforms,
	for each $x>0$,
	\begin{equation}
		\label{eq: pointwise}
		\mathcal{L}_{\boldsymbol{Z}_{\times}}\lbrb{(t-x),\boldsymbol{s}e^{-\rho t}} :=
		\Ebb{\exp\lbrb{-\langle\boldsymbol{Z}_\times(t-x),\boldsymbol{s}e^{-\rho t}\rangle}}\xrightarrow[t\to \infty]{}\mathcal{L}_{\boldsymbol{v}{W}_{\times}}\lbrb{\boldsymbol{s}e^{-\rho x}}.
	\end{equation}
	By \eqref{eq: Laplace mu} and
	Theorem \ref{thm: multi},
	we get that
	\begin{equation}
		\label{eq: LZ}
		\mathcal{L}_{\boldsymbol{Z}}\lbrb{t,e^{-\rho t}\boldsymbol{s}}=
		1 + 
		\sum_{n \geq 1}
		\frac{(-1)^n}{n!}\int_{(0,t]^n}
		D(\boldsymbol{x})       \prod_{i=1}^n\lbrb{1-\mathcal{L}_{\boldsymbol{Z}_{\times}}(t-x_i,\boldsymbol{s}e^{-\rho t})} \Lambda(\D \boldsymbol{
			x
		}).
	\end{equation}
	From \eqref{eq: pointwise}, we have that
	for each $\boldsymbol{x}\in\R^n$,
	\begin{equation}
		\label{eq: pointwise integral}
		\ind{\boldsymbol{x} \in (0,t]^n}D(\boldsymbol{x})\prod_{i=1}^n\lbrb{1-\mathcal{L}_{\boldsymbol{Z}_{\times}}(t-x_i,\boldsymbol{s}e^{-\rho t})} 
		\xrightarrow[t \to \infty]{}
		D(\boldsymbol{x})
		\prod_{i=1}^n\lbrb{1-\mathcal{L}_{\boldsymbol{v}W_{\times}}(\boldsymbol{s}e^{-\rho x_i})}.
	\end{equation}
	To use the last result under the integral sign in \eqref{eq: LZ},
	we need to provide a
	bound with an integrable function. First, by
	Hadamard's inequality \eqref{eq:Hadamard}, 
$D(\boldsymbol{x})
		% =
		% \det(K(x_i,x_j))_{1\leq i,j \leq n}
		\leq \prod_{i=1}^n K(x_i,x_i).
$
	Second, note that as we are working with Laplace transforms of positive random variables, we can assume that the vector
	$\boldsymbol{s}$ has positive entries. Let $C_{\boldsymbol{s}}^\ast := \max_j s_j/\min_j v_j$ so that for
	all $j$, $0\leq s_j\leq C_{\boldsymbol{s}}^\ast v_j$. We next bound as in \cite[Section 3.1]{Rabehasaina-21},
	\begin{equation}
		\label{eq: boundeq}
		\begin{split}
			% \mathcal{L}_{\boldsymbol{Z}_{\times,1}}(t-x_i,\boldsymbol{s}e^{-\rho t})
			% =
			&e^{-\rho t}\Ebb{\langle \boldsymbol{Z}_{\times}(t-x), \boldsymbol{s} \rangle}
			=e^{-\rho x}\Ebb{\langle  e^{-\rho(t-x)}\boldsymbol{Z}_{\times}(t-x),\boldsymbol{s}\rangle}  \\
			% &\leq
			% e^{-\rho x}\Ebb{\langle |\boldsymbol{s}|, e^{-\rho(t-x)}\boldsymbol{Z}_{\times}(t-x)\rangle}\\
            &\qquad\qquad	\leq C_{\boldsymbol{s}}^\ast
			e^{-\rho x}\Ebb{\langle e^{-\rho(t-x)}\boldsymbol{Z}_{\times}(t-x),\boldsymbol{v}\rangle} \\
			 &\qquad\qquad=C_{\boldsymbol{s}}^\ast
			e^{-\rho x}\Ebb{\langle e^{-\rho\cdot0}\boldsymbol{Z}_{\times}(0),
				\boldsymbol{v}\rangle}
				=C_{\boldsymbol{s}}^\ast
			e^{-\rho x}\Ebb{\langle \boldsymbol{I}_{0}, \boldsymbol{v}\rangle}=: C_{\boldsymbol{s}}e^{-\rho x},
		\end{split}
	\end{equation}
	where we used the fact that $(\langle e^{-\rho t}\boldsymbol{Z}_{\times}(t),
	\boldsymbol{v}\rangle)_{t\geq0}$ is a martingale,
	see \cite[p.209]{Athreya-Ney-72}.
	 Therefore,
	\begin{align*}
		\labsrabs{1- \mathcal{L}_{\boldsymbol{Z}_{\times}}(t-x,\boldsymbol{s}e^{-\rho t})}
		&=
		\Ebb{1- \exp\lbrb{-e^{-\rho t}\langle\boldsymbol{Z}_{\times}(t-x),
				\boldsymbol{s}\rangle}}\leq 
		\Ebb{e^{-\rho t}\langle \boldsymbol{Z}_{\times}(t-x),\boldsymbol{s}\rangle}\\
		&\leq C_{\boldsymbol{s}}e^{-\rho x},
	\end{align*}
	where in the first inequality we used that $1-e^{-x}\leq x$ for positive $x$.
	Thus we get
	\begin{equation}
		\label{eq: abs bound}
		\labsrabs{
			\ind{\boldsymbol{x} \in (0,t]^n}D(\boldsymbol{x})\prod_{i=1}^n\lbrb{1-\mathcal{L}_{\boldsymbol{Z}_{\times}}(t-x_i,\boldsymbol{s}e^{-\rho t})} }
		\leq 
		C_{\boldsymbol{s}}^n\prod_{i=1}^n K(x_i,x_i)e^{-\rho x_i}.
	\end{equation}
	The last is indeed
	a $\Lambda(\D \boldsymbol{x})$-integrable function because 
     $\int_{(0,\infty)} e^{-\rho x} K(x,x)\Lambda(\D x)$ is finite by assumption, so 
	\begin{equation}
		\begin{split}
			\label{eq: bound K}
			&C_{\boldsymbol{s}}^n\int_{(0,\infty)^n} \prod_{i=1}^n\det(K(x_i,x_i)) e^{-\rho x_i} \Lambda(\D \boldsymbol{
				x
			}) 
			=
			\lbrb{C_{\boldsymbol{s}}\int_{(0,\infty)} e^{-\rho x} K(x,x)\Lambda (\D x) }^n<\infty.
		\end{split}
	\end{equation}
	Therefore, we
	can apply the dominated convergence theorem in
	\eqref{eq: pointwise integral}, which gives that, for each $n$,
	\begin{align*}
		\int_{(0,t]^n}
		D(\boldsymbol{x})     &  \prod_{i=1}^n\lbrb{1-\mathcal{L}_{\boldsymbol{Z}_{\times}}(t-x_i,\boldsymbol{s}e^{-\rho t})} \Lambda(\D \boldsymbol{
			x
		}) \\
		&
		\hspace{2cm}\xrightarrow[t \to \infty]{}
		\int_{(0,\infty)^n} D(\boldsymbol{x})       \prod_{i=1}^n\lbrb{1-\mathcal{L}_{\boldsymbol{v}W_{\times}}(\boldsymbol{s}e^{-\rho x_i})}
		\Lambda(\D \boldsymbol{
			x
		}) .    
	\end{align*}
	Substituting in \eqref{eq: LZ}, again by dominated convergence and the bound
	\eqref{eq: bound K}, we obtain that
	\[
	\mathcal{L}_{\boldsymbol{Z}}\lbrb{t,e^{-\rho t}\boldsymbol{s}}\xrightarrow[t \to \infty]{}
	1 + 
	\sum_{n \geq 1}
	\frac{(-1)^n}{n!}\int_{(0,\infty)^n}
	D(\boldsymbol{x})      \prod_{i=1}^n\lbrb{1-\mathcal{L}_{\boldsymbol{v}W_{\times}}(\boldsymbol{s}e^{-\rho x_i})} \Lambda(\D \boldsymbol{
		x
	}) .
	\]
	We note that the sum in the last limit is
	absolutely convergent by \eqref{eq: bound K}.
	Letting $\Phi$
	be the underlying
	DPP, we can rewrite the last as
	\[
	\mathcal{L}_{\boldsymbol{Z}}\lbrb{t,e^{-\rho t}\boldsymbol{s}}\xrightarrow[t \to \infty]{}
	\mathcal{L}_{\Phi}
	\lbrb{-\ln\lbrb{\mathcal{L}_{\boldsymbol{v}W_\times
			}\lbrb{\boldsymbol{s}e^{-\rho x}}}},
	\]
	which is exactly the Laplace transform of
	\[
	\int_{(0,\infty)} e^{-\rho x}\Phi^\ast(\D x),
	\]
	where $\Phi^\ast(\D x)$ is the random measure
	$\sum_k \beta_k \delta_{T_i}$, where
	$\beta_k$ are iid 
	$\R^d$ random variables
	with distribution
	$\boldsymbol{v}W_\times$
	and $T_i$ are the atoms
	of the determinantal point process $\Phi$,
	see \cite[Example 2.2.30]{Baccelli-Blaszczyszyn-Karray-24}.
	This concludes the proof of Theorem \ref{thm: DPP}.
\end{proof}

\subsection{Proof of Theorem \ref{thm: DPP2}}\label{sec: proof DPP2}

\begin{proof}[Proof of Lemma \ref{lemma 2}]	
We will use the following result. 
\begin{lemma}[Lemmas 2 and 3 from \cite{Bojkova-Hyrien-Yanev-22}]\label{lem: lemma 23}
	Let $\alpha_\infty$ and $\beta_\infty$ be positive real constants. 
    Then we have the following.
	\begin{enumerate}
		\item  If \( \alpha(t) \simi \alpha_\infty e^{\delta t} \) with \( \delta \in \mathbb{R} \) and 
		$
		\int_{0}^{\infty} e^{-\delta x}\beta(x)\, \D x $ is finite,
		then
		 \[ \int_0^t
         \beta(t-x)\alpha(x) \D x \simi \alpha(t) \int_{0}^{\infty} e^{-\delta x}\beta(x) \D x. \]
		
		\item  If \( \alpha(t) \simi \alpha_\infty\)  and $ \beta(t)  \simi \beta_\infty$,
		then 
		\[
\int_0^t
\beta(t-x)
\alpha(x) \D x\simi 
\alpha_\infty \beta_\infty t.
		\]
	\end{enumerate}
\end{lemma}

	Let us start with the observation that the stationarity assumption
	\eqref{eq: stationary}, applied for $n=1$, implies that $K(x,x)$ 
	is constant over $x$. Denote its value by $K_\ast$.
By the classical result, we have
	see e.g. \cite[p. 203]{Athreya-Ney-72} or \cite[p.151]{Sevastyanov-71}, 
	    \begin{equation}
\label{eq: expectaion a_i}
\Ebb{Z_{\times,i}(t)} \simi
	e^{\rho t} a_i:= e^{\rho t} \langle
	\boldsymbol{u}, \Ebb{\boldsymbol{I}}
	\rangle v_i
	,
    \end{equation}
Next, by Theorem
	\ref{thm: moments},
	$\Ebb{Z_i(t)} = K_\ast \int_0^t
	\Ebb{Z_{\times,i}(t-x)} \lambda(x)\D x$,
    so 
\begin{itemize}
    \item if $\delta>\rho$,
    by \eqref{eq: expectaion a_i} the integral
    $\int_0^\infty
    e^{-\delta x} \Ebb{Z_{\times,i}(x)}\D x$ is finite, and from the first part 	of  Lemma \ref{lem: lemma 23} with
	$\alpha(t) = \lambda(t)$ and $\beta(t) = \Ebb{Z_{\times,i}(t)}$, we obtain
\[
	\Ebb{Z_i(t)} \simi e^{\delta t} A_i:=
	e^{\delta t} K_\ast \lambda_\infty \int_0^\infty e^{-\delta x}\Ebb{Z_{\times,i}(x)}\D x.
\]
\item if $\delta = \rho$, we apply the second part of
Lemma \ref{lem: lemma 23} with
$\beta(t) = e^{-\delta t}\Ebb{Z_{\times,i}(t)}\simi a_i$,
and $\alpha(t) =  e^{-\delta t}\lambda(t) \simi\lambda_\infty$,
so
\[
\Ebb{Z_i(t)} = K_\ast e^{\delta t}\int_{0}^t 
\beta(t-x)\alpha(x)\D x \simi K_\ast e^{\delta t} t a_i \lambda_\infty =: A_i' t e^{\delta t}.
\]
\end{itemize}

We continue with a similar strategy for the variance:
the asymptotics of the second moment of $Z_{\times,i}(t)$
	are available in
	\cite[p.152]{Sevastyanov-71} or
	\cite[p.204]{Athreya-Ney-72},
	giving us,
	for some positive constants $B_{j,i}$,
	\[
	\Ebb{\lbrb{Z_{\times, i}(t)}^2}\simi
	\begin{cases}
		e^{\rho t} B_{1,i} & \text{if
			$\rho<0$} \\
		t B_{2,i}  & \text{if
			$\rho = 0$} \\
		e^{2\rho t}B_{3, i}  & \text{if $\rho > 0$}.
	\end{cases}
	\]
Further,  by the second part of Theorem \ref{thm: moments},    
	\[\Var(Z_i(t))
	\leq K_\ast \int_0^t \Ebb{\lbrb{Z_{\times, i}(t-x)}^2} \lambda(x) \D x.
	\]
	We consider the relevant cases.
    \begin{itemize}
        \item Fix $\delta>\rho$. We apply Lemma \ref{lem: lemma 23} with
        $\alpha(t) = \Ebb{(Z_{\times,i}(t))^2}$
        and $\beta(t) = \lambda(t) \simi e^{\delta t}\lambda_\infty$. Note that if $\rho<0$, $\alpha(t) \simi e^{\rho t}B_{1,i}$, and
        if $\rho = 0$, $\alpha(t) \simi tB_{2,i}$, so in both cases $\int_0^\infty e^{-\delta x}\alpha(x) \D x$ is finite, so we obtain
      \[   K_\ast \int_0^t \Ebb{\lbrb{Z_{\times, i}(t-x)}^2} \lambda(x) \D x \simi K_ \ast e^{\delta t} \lambda_{\infty}\int_0^\infty
      e^{-\delta x}\beta(x)\D x = \bo{e^{\delta t}}.
\]
        \item Fix $\rho > 0$ and $\delta = 2\rho$. We
        apply the second item of  Lemma \ref{lem: lemma 23} with
        $\alpha(t)
        = e^{-2\rho t}\Ebb{(Z_{\times,i}(t))^2}\simi B_{3,i}$
        and $\beta(t) = e^{-\delta t}\lambda(t)\simi 
        \lambda_\infty$, so
      \[   K_\ast \int_0^t \Ebb{\lbrb{Z_{\times, i}(t-x)}^2} \lambda(x) \D x = K_ \ast e^{2 \rho t} \int_0^t
      \alpha(t-x) \beta (t) \D t
       = \bo{te^{2 \rho t}}
        =\bo{te^{\delta t}}.
\]
\item With similar arguments, if $\delta>2\rho>0$, we have $\Ebb{(Z_{\times,i}(t))^2}\simi e^{2\rho t}B_{3,i}$, so
we apply Lemma \ref{lem: lemma 23} with
$\alpha(t) :=
\Ebb{(Z_{\times,i}(t))^2}$ and
$\beta(t) = \lambda(t)$. If $
\delta < 2\rho$,
we should reverse the roles of $\alpha$ and $\beta$.
        \item If $\delta=\rho>0$, we use Lemma \ref{lemma 2} with
        $\alpha(t) = \lambda(t)\simi \lambda_\infty e^{
        \delta x}$ and $\beta(t) = \Ebb{(Z_{\times,i}(t))^2 } \simi e^{2\rho t}B_{3,i}.$
    \end{itemize}
    \end{proof}
\begin{proof}[Proof of Theorem \ref{thm: DPP2}]
	
For items \ref{it:3} and \ref{it:4} we follow an approach similar
	to the one chosen by Rabehasaina and Woo in \cite{Rabehasaina-21}, and for \ref{it:1}  and
\ref{it:2},
	techniques from the works \cite{ Hyrien-Mitov-Yanev-13-super, Hyrien-Mitov-Yanev-15-sub,Bojkova-Hyrien-Yanev-22,Bojkova-Yanev-19-crit}, which core argument lies in Lemma \ref{lemma 2}.

Indeed substituting its estimates, we see that if
$\delta > \max\{\rho,0\}$ or $\delta=\rho>0$,
\begin{equation}
		\label{eq: L2 limit}
		\Ebb{\lbrb{\frac{Z_i(t)}{\Ebb{Z_i(t)}} - 1}^2}  = \frac{\Var(Z_i(t))}{\Ebb{Z_i(t)}^2}\xrightarrow[t\to\infty]{}0,
	\end{equation}
   which gives exactly items \ref{it:1} and \ref{it:2} of the theorem.

	\textit{(\ref{it:4})} In this case we assume that
	the process is subcritical ($\rho<0$)
	and that $\lambda(x) 
	\simi \lambda_\infty$.
	As we assume that the immigration is spanned by a DPP, by Theorem \ref{thm: multi}
	\begin{equation}
		\label{eq: LZ thm2}
		\begin{split}
			\mathcal{L}_{\boldsymbol{Z}}\lbrb{t,\boldsymbol{s}}&=
			1 + 
			\sum_{n \geq 1}
			\frac{(-1)^n}{n!}\int_{(0,t]^n}
			D(\boldsymbol{x})       \prod_{i=1}^n\lbrb{1-\mathcal{L}_{\boldsymbol{Z}_{\times}}(t-x_i,\boldsymbol{s})} \Lambda(\D \boldsymbol{
				x
			}).
		\end{split}
	\end{equation}
	To exploit the convergence of $
	\lambda(t)$ as $t\to \infty$, we change variables $x_i \to t-x_i$ in the last
	integral. Note that \begin{equation}\label{eq:D bound stationary}
		\det(K(t-x_i,t-x_j))_{1\leq i,j \leq n}=
		\det(K(x_i,x_j))_{1\leq i,j \leq n}
		\leq \prod_{i=1}^n K(x_i,x_i) = 
		K_\ast^n,
	\end{equation}
	where the first follows from our assumption
	that the DPP is
	stationary and also using  inequality
	\eqref{eq:Hadamard}. Next, for $\boldsymbol{x} = (x_1, \dots, x_n)$, denoting	
	\[
	\lambda( t- \boldsymbol{x}):=
	\lambda(t-x_1) \dots
	\lambda(t-x_n),\quad
	\text{and}
	\quad
	\D \boldsymbol{x} := 
	\D x_1 \dots \D x_n,
	\]
	we get that
	\begin{equation}
		\label{eq: LZ 2.1}
		\begin{split}
			\mathcal{L}_{\boldsymbol{Z}}\lbrb{t,\boldsymbol{s}}&=
			1 + 
			\sum_{n \geq 1}
			\frac{(-1)^n}{n!}\int_{(0,t]^n}
			D(\boldsymbol{x})       \prod_{i=1}^n\lbrb{1-\mathcal{L}_{\boldsymbol{Z}_{\times}}(x_i,\boldsymbol{s})} \lambda(t- \boldsymbol{
				x
			})
			\D \boldsymbol{
				x}.
		\end{split}
	\end{equation}
	To apply the dominated 
	convergence theorem,
	first bound as in 
	\eqref{eq: boundeq}
	and with the same
	$C_{\boldsymbol{s}}$ as there,
	\begin{equation}
		\label{eq: 1-lz}
		\labsrabs{1- \mathcal{L}_{\boldsymbol{Z}_{\times}}(x,\boldsymbol{s})}
        =
		\Ebb{1- \exp\lbrb{-\langle\boldsymbol{Z}_{\times}(x),
				\boldsymbol{s}\rangle}}
		\leq
		\Ebb{\langle \boldsymbol{Z}_{\times}(x), \boldsymbol{s} \rangle} \leq 
		e^{\rho x}C_{\boldsymbol{s}}.
	\end{equation}
	Next, pick $C$ such that
	$\sup_{t \geq C} \lambda(t)<
	\infty$, which exists
	as $\lambda$ has a finite limit at infinity. Therefore,
	for each $\boldsymbol{x}\in\R^n$,
	\begin{equation*}
		\begin{split}
			&\labsrabs{
				\ind{\boldsymbol{x} \in (0,t]^n}D(\boldsymbol{x})\prod_{i=1}^n\left(1-\mathcal{L}_{\boldsymbol{Z}_{\times}}(x_i,\boldsymbol{s})\right) \lambda(t-\boldsymbol{x})
			} 
			\leq 
			\prod_{i=1}^n
			\ind{x_i \in (0,t)}K_\ast C_{\boldsymbol{s}}e^{
				\rho x_i}
			\lambda(t-x_i)\\
			&\hspace{1cm}=
			K_\ast^n C_{\boldsymbol{s}}^n
			\prod_{i=1}^n \lbrb{\ind{x_i \in (0,t-C)}e^{
					\rho x_i}
				\lambda(t-x_i)
				+
				\ind{x_i \in (t-C,t)}e^{
					\rho x_i}
				\lambda(t-x_i) }
			\\
			&\hspace{1cm}\leq
			K_\ast^n C_{\boldsymbol{s}}^n
			\prod_{i=1}^n
			\lbrb{e^{
					\rho x_i}
				\sup_{x\geq C} \lambda(x)
				+
				\ind{x_i \in (0,C)}
				\lambda(x_i) }.
		\end{split}
	\end{equation*}
	The last is a $\D \boldsymbol{x}$-integrable function with
	\begin{equation}
		\label{eq: geom 2.1}
		\begin{split}
			\int_{(0,\infty)^n}
			K_\ast^n C_{\boldsymbol{s}}^n
			& \prod_{i=1}^n
			\lbrb{e^{
					\rho x_i}
				\sup_{x\geq C} \lambda(x)
				+
				\ind{x_i \in (0,C)}
				\lambda(x_i) } \D \boldsymbol{x}\\
			& =
			\lbrb{K_ \ast C_{\boldsymbol{s}}\int_0^\infty 
				\lbrb{e^{
						\rho x_i}
					\sup_{x\geq C} \lambda(x)
					+
					\ind{x_i \in (0,C)}
					\lambda(x) } \D x}^n < 
			\infty,
		\end{split}
	\end{equation}
	because by assumption the
	directing measure $\Lambda$ of a DPP  is always $\sigma$-finite and because in this case $\rho<0$. Therefore
	by dominated convergence
	\begin{align*}
		\int_{(0,t)^n}
		D(\boldsymbol{x})       \prod_{i=1}^n&\lbrb{1-\mathcal{L}_{\boldsymbol{Z}_{\times}}(x_i,\boldsymbol{s})} \lambda(t- \boldsymbol{
			x
		})
		\D \boldsymbol{
			x}\\
		&\xrightarrow[t \to \infty]
		{}
		\lambda_\infty^n\int_{(0,\infty)^n}
		D(\boldsymbol{x})       \prod_{i=1}^n\lbrb{1-\mathcal{L}_{\boldsymbol{Z}_{\times}}(x_i,\boldsymbol{s})}
		\D \boldsymbol{
			x},
	\end{align*} 
	and again by dominated convergence, which we can apply
	thanks to the geometric bound
	\eqref{eq: geom 2.1},
	for the sum in \eqref{eq: LZ 2.1}, we get
	\begin{equation}
		\label{eq:Laplace limit}
		\mathcal{L}_{\boldsymbol{Z}}\lbrb{t,\boldsymbol{s}}
		\xrightarrow[t \to \infty]{}
		1 + 
		\sum_{n \geq 1}
		\frac{(-1)^n\lambda^n_{\infty}
		}{n!}\int_{(0,\infty)^n}
		D(\boldsymbol{x})       \prod_{i=1}^n\lbrb{1-\mathcal{L}_{\boldsymbol{Z}_{\times}}(x_i,\boldsymbol{s})}
		\D \boldsymbol{
			x}.
	\end{equation}
	To conclude that the last is the Laplace transform of a positive random variable, it remains to 
	verify tends to $1$ as $\boldsymbol{s}\to \boldsymbol{0}$,
	see e.g. \cite[p. 431, Chapter XIII.1, Theorem 2]{Feller-71}.
	This is indeed true by the bound
	\eqref{eq: 1-lz} and the choice
	$C_{\boldsymbol{s}} = \max_j{s_j}/\min_j{v_j}$.  
	
	\textit{(\ref{it:3})} 
    For critical processes without immigration, Holte \cite{Holte-82} obtains the limit distribution under optimal assumptions and for a more general class of branching mechanisms. Note, however, that Holte’s $\boldsymbol v$ is an eigenvector of the matrix with entries $
  \Ebb{(\nu_i)_j},
$
whereas in this work we define~$\boldsymbol v$ via \eqref{eq:def matrix A} (see also Remark~\ref{rem: eigen}).
Denoting with
 lower index $H$ the corresponding quantities from the work of Holte, and taking into account the normalization there
 $\langle\boldsymbol{u}_H, \boldsymbol{v}_H \rangle = 1,
     \langle\boldsymbol{v}_H, \boldsymbol{1} \rangle = 1$, we have that for
     $D_1:={\sum _j(v_j/\mu_j) 
	}$ and
    $D_2:=\sum_j (u_j v_j/\mu_j)$,
     \begin{equation}
         \label{eq: Holte eigen}
u_{H,i} = \frac{D_1
	}{D_2}u_i,
	\quad
	v_{H,i} = \frac{1}{D_1}\frac{v_i}{\mu_i},
	\quad
	\text{and}
	\quad
	\sum_j \mu_j u_{H,j}v_{H,j}= \frac{1}{D_2}.
      \end{equation}
We show how the results of Holte \cite{Holte-82} translate to our setting.
Let $(e_i)$ be the standard basis in $\R^d$. 
 Therefore from \cite[Lemma on page p.493]{Holte-82}, and using lower indices $H$ for constants from \cite{Holte-82},
 we have that
     \begin{align*}
 H_i\lbrb{\frac{\boldsymbol{s}}{t},t} &:=
 \Ebb{1-\exp\left.\lbrb{-\langlerangle{Z_\times(t),\frac{\boldsymbol{s}}{t}}}\right|
 Z_{\times}(0)=e_i)}\\
 &=
        \frac{1}{t}\frac{||\boldsymbol{s||}_H}{1+\langlerangle{\boldsymbol{c}_H,{\boldsymbol{s}}}}\boldsymbol{u}_{H,i} + \so{\frac1t}
        =: \frac{C_i}{t}+ \so{\frac1t}.
	\end{align*}
Our process starts with a random number of particles of law $\boldsymbol{I}$, so
a first order expansion gives us
    \begin{equation}\label{eq: expansion}
    \begin{split}
1-\mathcal{L}_{\boldsymbol{Z}_{\times}}\lbrb{t,\frac{\boldsymbol{s}}{t}}
&= \Ebb{1- \exp\lbrb{-\langlerangle{\boldsymbol{Z}_\times(t),\frac{\boldsymbol{s}}{t}}}}\\
&=\sum_{\boldsymbol{x}=(x_1, \dots, x_d)}
\P(\boldsymbol{I}=\boldsymbol{x})\lbrb{1 - \prod_i {\lbrb{1-H_i\lbrb{\frac{\boldsymbol{s}}{t},t}}}^{x_i}}\\
&=\sum_{\boldsymbol{x}=(x_1, \dots, x_d)}
\P(\boldsymbol{I}=\boldsymbol{x})\lbrb{1 - \prod_i \lbrb{1-x_i\frac{C_i}{t}+\so{\frac1t}}}\\
&=
\sum_{\boldsymbol{x}=(x_1, \dots, x_d)}
\P(\boldsymbol{I}=\boldsymbol{x})\lbrb{\sum_i \lbrb{x_i\frac{C_i}{t}+\so{\frac1t}}}\\
&=\frac{1}{t}\Ebb{\langlerangle{\boldsymbol{I},{\boldsymbol{C}}}} +\so{\frac1t},
\quad
\text{with }
\boldsymbol{C}:=(C_1, \dots, C_d).
    \end{split}
    \end{equation}
    Setting our parameters
	\[
	\beta :=
	\frac{\Ebb{\langle\boldsymbol{I},\boldsymbol{u}\rangle}}{Q} 
	%  \sum_{i=1}^d
	% \Ebb{\boldsymbol{I}_j}{\boldsymbol{u}_j}
	\quad
	\text{with}
	\quad
	Q:=\frac{1}{2}\sum_{i,j,k=1}^d
	\left.\frac{\partial^2 G_{{\nu_i}}}{\partial x_j \partial x_k}\right|_{\boldsymbol{x}=\boldsymbol{1}}
	\mu_i^{-1}{v}_i{u}_j
	{u}_k,
	\]
    using \eqref{eq: Holte eigen}, we calculate
    the constants from \cite{Holte-82}:
    \[
    \beta_H = \frac{1}{D_2}, \quad
  \frac{1}{2}\langlerangle{\boldsymbol{v}_H,\boldsymbol{q[u]}}
=
\frac{D_1}{D_2^2}
Q, \quad  ||\boldsymbol{s}||_H = \frac{D_2}{D_1}\langlerangle{\boldsymbol{s},
\boldsymbol{v}}, \quad 
\text{and}
\quad
\boldsymbol{c}_H = 
Q \boldsymbol{v}.
    \]
Substituting the definition of $C_i$ in
    \eqref{eq: expansion},
    we obtain
        	\[
	t\lbrb{1-\mathcal{L}_{\boldsymbol{Z}_{\times}}\lbrb{t,\frac{\boldsymbol{s}}{t}}}
	\xrightarrow[t\to\infty]{}
	%     \sum_{i=1}^d
	% \Ebb{\boldsymbol{I}_i}\boldsymbol{v}_i
	{\beta}
	\frac{\langle\boldsymbol{s},
		\boldsymbol{v}
		\rangle}{
		1/Q+\boldsymbol{\langle s,
			\boldsymbol{v}
			\rangle}  
	}.
	\]
    We note that $Q\neq0$: as earlier, see  \cite[p.203]{Athreya-Ney-72},
    $u_i$ and $v_i$ are strictly positive, and because the derivatives represent factorial moments of $\Nb$-valued random variables, if $Q=0$ then $\Ebb{\nu_i(\nu_i-1)}=0$ for all $i$,
    so $\nu_i$ can take values only $0$ and 1. However by the criticality assumption, $\Ebb{Z_\times(t)}=\Ebb{Z_\times(0)}$, so we must have $\P(\nu_i =1)=1$, so $\Var(||\boldsymbol{\nu||})=0$, which is not possible by our non-degeneracy assumption.
    
We will use the last limit in a form which is obtained by substituting $t$ with $t(1-x)$, and $\boldsymbol{s}$
	with $\boldsymbol{s}(1-x)$ for $x \in (0,1)$, which gives
	\[
	t(1-x)\lbrb{1-\mathcal{L}_{\boldsymbol{Z}_{\times}}\lbrb{t(1-x),\frac{\boldsymbol{s}}{t}}}
	\xrightarrow[t\to\infty]{}
	%     \sum_{i=1}^d
	% \Ebb{\boldsymbol{I}_i}\boldsymbol{v}_i \overline{\beta}
	\beta\frac{(1-x)\langle\boldsymbol{s,
			\boldsymbol{v}
		}\rangle}{
		1/Q+(1-x)\boldsymbol{\langle s,
			\boldsymbol{v}
			\rangle}   
	}.
	\]
	Dividing by $(1-x)$, then substituting in \eqref{eq: LZ thm2}
	and changing variables $x_i \to tx_i$, we get
	\[
	\mathcal{L}_{\boldsymbol{Z}}\lbrb{t,\frac{\boldsymbol{s}}{t}}=
	1 + 
	\sum_{n \geq 1}
	\frac{(-1)^n}{n!}\int_{(0,1)^n}
	D(t\boldsymbol{x})       \prod_{i=1}^nt\lbrb{1-\mathcal{L}_{\boldsymbol{Z}_{\times}}\lbrb{t(1-x_i),\frac{\boldsymbol{s}}{t}}} \lambda\lbrb{t\boldsymbol{x}}\D \boldsymbol{
		x
	}
	\]
	Next,   \eqref{eq:D bound stationary} and \eqref{eq: 1-lz} give constant bounds of $D(t\boldsymbol{x})$ and
    the product. Moreover, as $\lambda$ is integrable, it is bounded on $|x|>\epsilon$. We also have that
	$D(t\boldsymbol{x})$
	tends to 1 $\D \boldsymbol{x}$-a.e., because for $\boldsymbol{x}$ with different entries, $K(tx_i,tx_j) = K(t|x_i-x_j|,0)$ tends to 0 if $x_i \neq x_j$ for $t \to \infty$ by assumption. Therefore by dominated convergence,
	\begin{align*}
		\mathcal{L}_{\boldsymbol{Z}}\lbrb{t,\frac{\boldsymbol{s}}{t}}
		&\xrightarrow[t \to \infty]{}
		1 + 
		\sum_{n \geq 1}\frac{(-1)^n}{n!}
		\lbrb{K^\ast\int_0^1\lambda_\infty
			\beta
			\frac{\langle\boldsymbol{s},
				\boldsymbol{v}
				\rangle}{
				1/Q+(1-x)\boldsymbol{\langle s,
					\boldsymbol{v}
					\rangle}}\D x}^n\\
		&=
		\lbrb{\frac{1/Q}{1/Q+\langle\boldsymbol{s},
				\boldsymbol{v}\rangle}}^{K^\ast\lambda_\infty\beta}.
	\end{align*}
	The last is indeed  the Laplace transform of
	$Y\boldsymbol{v}$ with
	$Y\sim \Gamma(K ^\ast\lambda_\infty \beta,1/Q)$ as claimed, which concludes the proof.
    	
\end{proof}

\section*{Acknowledgments}
We gratefully thank the anonymous referees for their careful review of our manuscript and for their valuable comments and suggestions.

This study is financed by the European Union's NextGenerationEU, through the National Recovery and Resilience Plan of the Republic of Bulgaria, project No BG-RRP-2.004-0008.

\addcontentsline{toc}{section}{References}

\bibliographystyle{plain}
\bibliography{bibliography.bib}

\end{document}